\documentclass[11pt]{article}
\usepackage{enumerate}
\usepackage{amssymb,amsmath,amsfonts,amsthm,mathrsfs}
\usepackage[colorlinks=true, pdfstartview=FitV, linkcolor=blue, citecolor=blue, urlcolor=blue]{hyperref}
\usepackage{graphicx,color}
\usepackage{cite}
\usepackage{indentfirst}

\allowdisplaybreaks

\usepackage{latexsym}
\usepackage{fancyhdr}
\allowdisplaybreaks \allowdisplaybreaks[4]
 \usepackage{comment}
\newtheorem{thm}{Theorem}[section]

\theoremstyle{definition}

\numberwithin{equation}{section}

\begin{document}


\baselineskip=17pt

\renewcommand{\thefootnote}{\fnsymbol {footnote}}

\title{\bf\Large Sharp bounds for  the multilinear integral operators  on Heisenberg group BMO space}
\footnotetext {{}{2010 \emph{Mathematics Subject Classification}: Primary 43A15; Secondary 42B35 26D15}} \footnotetext {{}\emph{Key words and phrases}: Sharp bounds, Hardy-Littlewood-P\'{o}lya operator, Heisenberg group BMO space}\setcounter{footnote}{0}\author{Huan Liang,\ Xiang Li\footnote{Corresponding author, \textsf{lixiang162@mails.ucas.ac.cn}},\ Dunyan Yan}
\date{}
\maketitle



\begin{abstract}
	In this paper, we study $m$-linear $n$-demensional  Hardy-Littlewood-P\'{o}lya operator and $m$-linear $n$-demensional  Hilbert operator on Heisenberg group BMO space. We obtain that the above two $m$-linear $n$-demensional operators is bounded in the BMO space of the Heisenberg group.
\end{abstract}

\newtheorem{theorem}{Theorem}[section]
\newtheorem{preliminaries}{Preliminaries}[section]
\newtheorem{definition}{Definition}[section]
\newtheorem{main result}{Main Result}[section]
\newtheorem{lemma}{Lemma}[section]
\newtheorem{proposition}{Proposition}[section]
\newtheorem{corollary}{Corollary}[section]
\newtheorem{remark}{Remark}[section]

\section[Introduction]{Introduction}

The Heisenberg group is a very typical non-commutative group, and its research on up-modulation and analytic problems is an extension and development of Euclidean space upharmonic and analytical problems and is an important part of non-commutative harmonic analysis \cite{cou} and \cite{tha}.

In this paper, we will prove that Heisenberg group BMO space are useful when we consider norm inequalities of  $m$-linear $n$-demensional operator, see \cite{chu} and \cite{zha}.

Let us introduce some basic knowledge about Heisenberg group. The Heisenberg group $\mathbb{H}^n$ is a non-commutative nilpotent Lie group, with the underlying manifold $\mathbb{R}^{2n}$$\times$$\mathbb{R}$ and the group law.

Let
$$
x=(x_1,\dots,x_{2n},x_{2n+1}),y=(y_1,\dots,y_{2n},y_{2n+1}),
$$
then
$$
x\times y=(x_1+y_1,\dots,x_{2n}+y_{2n},x_{2n+1}+y_{2n+1}+2\sum_{j=1}^n(y_jx_{n+j}-x_jy_{n+j}).
$$

The Heisenberg group $\mathbb{H}^n$ is a homogeneous group with dilations
$$
\delta_r(x_1,x_2,\dots,x_{2n},x_{2n+1})=(rx_1,rx_2,\dots,rx_{2n},r^2x_{2n+1}), r>0.
$$

The Haar measure on $\mathbb{H}^n$ coincides with the usual Lebesgue measure on $\mathbb{R}^{2n+1}$. We denote any mrasurable set $E\subset \mathbb{H}^n$ by $\left|E\right|$, then
$$
\left| \delta_r \left(E\right) \right|=r^Q\left|E\right|,d(\delta_rx)=r_Qdx,
$$
where $Q=2n+2$ is called the homogeneous dimension of $\mathbb{H}^n$.

The Heisenberg distance derived from the norm
$$
| x |_h=\left[ \left( \sum_{i=1}^{2n}{x_{i}^{2}} \right) ^2+x_{2n+1}^{2} \right]^{{{1}/{4}}},
$$
where $x=(x_1,x_2,\dots,x_{2n},x_{2n+1})$ is given by
$$
d(p,q)=d(q^{-1}p,0)=|q^{-1}p|_h.
$$
This distance $d$ is left-invariant, it meaning that $d(p,q)$ remains constant when both $p$ and $q$ are left shifted by some fixed vector on $\mathbb{H}^n$. Furthermore, $d$ satisfies the trigonometric inequality defined by \cite{kor}
$$
d(p,q)\leq {d(p,x)+d(x,q)},\quad p,x,q\in \mathbb{H}^n.
$$
For $r>0$ and $x\in \mathbb{H}^n$, the ball and sphere with center $x$ and radius $r$ on $\mathbb{H}^n$ are given by
$$
B(x,r)={y\in \mathbb{H}^n :d(x,y)<r }
$$
and
$$
S(x,r)={y\in \mathbb{H}^n :d(x,y)=r }.
$$
Then we obtain
$$
|B(x,r)|=|B(0,r)|=\varOmega_Qr^Q,
$$
where
$$
\varOmega_Q=\frac{2\pi^{n+\frac{1}{2}}\varGamma \left({{n}/{2}} \right)}{(n+1) \varGamma
	(n) \varGamma((n+1)/2)}
$$
represents the volume of the unit sphere $B(0,1)$ on $\mathbb{H}^n$, and $\omega _Q=Q\varOmega _Q$ (see \cite{cou}). More details about the Heisenberg group can be found in \cite{fol}, \cite{tha} and \cite{xiao}.

Next, we will provide some definitions used in the follows.

The BMO$(\mathbb{R}^n)$  consist of all measurable functions $f$ satisfying
$$
\left\|f \right\|_{BMO( \mathbb{R}^n )}:=\underset{B\subset \mathbb{R}^n}\sup\frac{1}{| B |}\int_B{| f(x)-f_B |dx<\infty}.
$$
The functional $\left\| \cdot \right\|_{BMO( \mathbb{R}^n )}$ is a norm on the spaces which is $\mathrm{d}x$ almost everywhere equal \cite{ben}.
BMO spaces on various domains frequently appear in the literature and are quite suitable for investigations, see \cite{ben}.

In this paper, we will define the the Heisenberg group BMO space as follows.

\begin{definition}
	The bounded mean oscillation space BMO($\mathbb{H}^n$) is defined to be the space of all locally integrable functions $f$ on $\mathbb{H}^n$ such that
	$$
	\left\|f \right\|_{BMO( \mathbb{H}^n )}:=\underset{B\subset \mathbb{H}^n}\sup\frac{1}{| B |}\int_B{| f( x )-f_B |dx<\infty},
	$$
	where the supremum is taken over all balls in $\mathbb{H}^n$ and
	$$
	f_B=\frac{1}{| B |}\int_B{f( x ) dx}.
	$$
\end{definition}

In \cite{ben}, The Hardy-Littlewood-P\'{o}lya operator is defined by
$$
H f(x)=\int_0^{+\infty} \frac{f(y)}{\max (x, y)} d y .
$$
Then, defined the  Heisenberg group $n$-dimensional Hardy-Littlewood-P\'{o}lya operator as follows
$$
H f(x)=\int_{\mathbb{H}^n} \frac{f(y)}{\max (|x|_{h}^{n}, |y|_{h}^{n})} d y.
$$
As a multilinear generalization of Calder\'{o}n operator, the $m$-linear $n$-dimensional Hardy-Littlewood-P\'{o}lya operator is defined by
$$
H_m( f_1,\dots,f_m)t(x)=\int_{\mathbb{R}^{nm}}{\frac{f_1( y_1 ) \cdots f_m( y_m )}{\max ( | x |^n,| y_1 |^n,\dots,| y_m |^n )^m} dy_1\cdots dy_m},
$$
where $x\in R^n\backslash\left\{ 0\right\}$ and $f_1,\dots,f_m$ are nonnegative locally integrable functions on $\mathbb{R}^n$. \\
The $m$-linear $n$-dimensional Hilbert operator  is defined by
$$
T_{m}( f_1,\dots,f_m) ( x )=\int_{\mathbb{R}^{nm}}{\frac{f_1( y_1 ) \cdots f_m( y_m )}{( | x |^{n}+| y_1 |^{n}+\cdots +| y_m |^{n} ) ^m}dy_1\cdots dy_m}.
$$
Similarly,  Fu et al. \cite{cou} defined the $m$-linear Heisenberg group Hardy-Littlewood-P\'{o}lya's operator as follows.
\begin{definition}
	Let $m$ be a positive integer and $f_1,\dots,f_m$ be nonnegative locally integrable
	functions on $\mathbb{H}^n$. The $m$-linear $n$-dimensional Hardy-Littlewood-P\'{o}lya operator  is defined by
	\begin{equation}
		H_{m}( f_1,\dots,f_m) ( x )=\int_{\mathbb{H}^{nm}}{\frac{f_1( y_1 ) \cdots f_m( y_m)}{\max ( | x |_{h}^{n},| y_1 |_{h}^{n},\dots,| y_m |_{h}^{n})^m}dy_1\cdots dy_m},
	\end{equation}
	where $x\in \mathbb{H}^n\backslash\left\{ 0\right\}$.
\end{definition}
\begin{definition}
	Let $m$ be a positive integer and $f_1,\ldots,f_m$ be nonnegative locally integrable
	functions on $\mathbb{H}^n$. The $m$-linear $n$-dimensional Hilbert operator  is defined by
	\begin{equation}
		T_{m}( f_1,\dots,f_m) ( x)=\int_{\mathbb{H}^{nm}}{\frac{f_1( y_1 ) \cdots f_m( y_m )}{( | x |_{h}^{n}+| y_1 |_{h}^{n}+\cdots +| y_m |_{h}^{n} ) ^m}dy_1\cdots dy_m},
	\end{equation}
	where $x\in \mathbb{H}^n\backslash\left\{ 0\right\}$.
\end{definition}

\begin{thm}\label{main_1}
	Let $m\in \mathbb{N}$, if
	
	\begin{equation}\label{condition_1}
		A_{m}^{h}:=\int_{\mathbb{H}^{nm}}
		\frac{1}{[\max(1,|y_{1}|^{n}_{h},\ldots,|y_{m}|^{n}_{h})]^{m}}dy_{1}\cdots dy_{m}<\infty,
	\end{equation}
	then $H_{m}$ is bounded from $BMO(\mathbb{H}^{n})\times\cdots \times BMO(\mathbb{H}^{n})$ to $BMO(\mathbb{H}^{n})$.
	
	Moreover, if condition $\eqref{condition_1}$ is satisfied, then the following formula for the norm of the operator holds
	\begin{equation}\label{sharp_bound_1}
		\|H_{m}\|_{\prod\limits_{j=1}^{m}BMO(\mathbb{H}^{n})\rightarrow BMO(\mathbb{H}^{n})}=A_{m}^{h}.
	\end{equation}
\end{thm}
\begin{thm}\label{main_11}
	Let $m\in \mathbb{N}$, if
	
	\begin{equation}\label{condition_11}
		B_{m}^{h}:=\int_{\mathbb{H}^{nm}}{\frac{dy_1\cdots dy_m}{( 1+| y_1 |_{h}^{n}+\cdots +| y_m |_{h}^{n} ) ^m}}< \infty,
	\end{equation}
	then $T_{m}$ is bounded from $BMO(\mathbb{H}^{n})\times\cdots \times BMO(\mathbb{H}^{n})$ to $BMO(\mathbb{H}^{n})$.
	
	Moreover, if condition $\eqref{condition_11}$ is satisfied, then the following formula for the norm of the operator holds
	\begin{equation}\label{sharp_bound_11}
		\|T_{m}\|_{\prod\limits_{j=1}^{m}BMO(\mathbb{H}^{n})\rightarrow BMO(\mathbb{H}^{n})}=B_{m}^{h}.
	\end{equation}
\end{thm}
The result of Theorem $\ref{main_1}$ and $\ref{main_11}$ are correct.

Now, we begin to prove our result.

\section{Proof of the main results}

\begin{proof}[Proof of Theorem $\ref{main_1}$and $\ref{main_11}$ ]
	
	Since the proof of the case when $m=1$ is similar and even simpler than that of the case when $m\geq 2$. For convenience, it is need to make a change of variable, then it follows that
	\begin{equation}\label{identity}
		H_{m}(f_{1},\ldots,f_{m})(x)=\int_{\mathbb{H}^{n}}\cdots \int_{\mathbb{H}^{n}}\frac{f_{1}(|x|_{h}^{-1}y_{1})\cdots f_{m}(|x|_{h}^{-1}y_{m})}{[\max(1,|y_{1}|^{n}_{h},\ldots,|y_{m}|^{n}_{h})]^{m}}dy_{1}\cdots dy_{m}
	\end{equation}
	and
	\begin{equation}\label{identity 1}
		T_{m}(f_{1},\ldots,f_{m})(x)=\int_{\mathbb{H}^{nm}}{\frac{f_1( \delta _{| x |_h}y_1 ) \cdots f_m( \delta _{| x |_h}y_m )}{( 1+| y_1 |_{h}^{n}+\cdots +| y_m |_{h}^{n} ) ^m}dy_1\cdots dy_m}.
	\end{equation}
	Using the above identity $\eqref{identity}$, $\eqref{identity 1}$ and the definition of Heisenberg group  $BMO(\mathbb{H}^{n})$, when $m=1$, we can infer that
	\begin{align*}
		(Hf(x))_B
		=&\frac{1}{|B|}\int_B{\int_{\mathbb{H}^n}{\frac{f( \delta _{|x|_h}y )}{\max ( 1,|y|_{h}^{n})}dy}dx}\\
		=&\int_{\mathbb{H}^n}{\left(\frac{1}{|B|}\int_B{f(\delta_{|x|_h}y)dx}\right)\frac{1}{\max(1,|y|_{h}^{n})}dy}\\
		=&\int_{\mathbb{H}^n}{\left(\frac{1}{|B|}\int_B{f(\delta_{|y|_h}x)dx}\right)\frac{1}{\max(1,|y|_{h}^{n})}dy}\\
		=&\int_{\mathbb{H}^n}{\left(\frac{1}{||y|_hB|}\int_{|y|_hB}{f(y)dy}\right)\frac{1}{\max(1,|y|_{h}^{n})}dy}\\
		=&\int_{\mathbb{H}^n}{f_{|y|_hB}\frac{1}{\max(1,|y|_{h}^{n})}dy}.
	\end{align*}
	Similarly, we obtain
	\begin{align*}
		(Tf(x))_B
		=&\frac{1}{|B|}\int_B{\int_{\mathbb{H}^n}{\frac{f( \delta _{|x|_h}y )}{( 1+| y |_{h}^{n} ) }dy}dx}\\
		=&\int_{\mathbb{H}^n}{\left(\frac{1}{|B|}\int_B{f(\delta_{|x|_h}y)dx}\right)\frac{1}{( 1+| y |_{h}^{n} ) }dy}\\
		=&\int_{\mathbb{H}^n}{\left(\frac{1}{|B|}\int_B{f(\delta_{|y|_h}x)dx}\right)\frac{1}{( 1+| y |_{h}^{n} ) }dy}\\
		=&\int_{\mathbb{H}^n}{\left(\frac{1}{||y|_hB|}\int_{|y|_hB}{f(y)dy}\right)\frac{1}{( 1+| y |_{h}^{n} ) }dy}\\
		=&\int_{\mathbb{H}^n}{f_{|y|_hB}\frac{1}{( 1+| y|_{h}^{n} ) }dy}.
	\end{align*}
	Then, using the Fubini theorem again, we get
	\begin{align*}
		\left\|Hf(x)\right\|_{BMO(\mathbb{H}^n)}
		=&\frac{1}{|B|}\int_B{|f(x)-(f(x))_B|dx}\\
		=&\frac{1}{|B|}\int_B{\left|\int_{\mathbb{H}^n}{\frac{f(\delta_{|x|_h}y)}{\max(1,|y|_{h}^{n})}}-\int_{\mathbb{H}^n}{f_{|y|_hB}\frac{1}{\max(1,|y|_{h}^{n})}dy}\right|dx}\\
		=&\frac{1}{|B|}\int_B{\left|\int_{\mathbb{H}^n}{\frac{f(\delta_{|x|_h}y)-f_{|y|_hB}}{\max(1,|y|_{h}^{n})}dy}\right|dx}\\
		\leq& \frac{1}{|B|}\int_B{\left(\int_{\mathbb{H}^n}{\left|\frac{f(\delta_{|x|_h}y)-f_{|y|_hB}}{\max(1,|y|_{h}^{n})}\right|dy}\right)dx}\\
		\leq& \frac{1}{|B|}\int_{\mathbb{H}^n}{\left(\int_B{\left|\frac{f(\delta_{|x|_h}y)-f_{|y|_hB}}{\max(1,|y|_{h}^{n})}\right|dx}\right)dy}\\
		=&\frac{1}{|B|}\int_{\mathbb{H}^n}{\left(\int_B{\left|f(\delta_{|x|_h}y)-f_{|y|_hB}\right|dx}\right)\frac{1}{\max(1,|y|_{h}^{n})}dy}\\
		=&\int_{\mathbb{H}^n}{\left(\frac{1}{\left|B\right|}\int_B{\left|f(\delta_{|x|_h}y)-f_{|y|_hB}\right|dx}\right)\frac{1}{\max(1,|y|_{h}^{n})}dy}\\
		=&\int_{\mathbb{H}^n}{\left(\frac{1}{\left|B\right|}\int_B{\left|f(\delta_{|y|_h}x)-f_{|y|_hB}\right|dx}\right)\frac{1}{\max(1,|y|_{h}^{n})}dy}\\
		=&\int_{\mathbb{H}^n}{\left(\frac{1}{\left|\left|y\right|_hB\right|}\int_{|y|_hB}{\left|f(y)-f_{|y|_hB}\right|dy}\right)\frac{1}{\max(1,|y|_{h}^{n})}dy}\\
		\leq& \left\| f \right\| _{BMO( \mathbb{H}^n )}\int_{\mathbb{H}^n}{\frac{1}{\max( 1,|y|_{h}^{n})}dy}
	\end{align*}
	and
	
	\begin{align*}
		\left\|Tf(x)\right\|_{BMO(\mathbb{H}^n)}
		=&\frac{1}{|B|}\int_B{\left|f(x)-(f(x))_B\right|dx}\\
		=&\frac{1}{|B|}\int_B{\left|\int_{\mathbb{H}^n}{\frac{f(\delta_{|x|_h}y)}{\left( 1+| y |_{m}^{n} \right)}}-\int_{H^n}{f_{|y|_hB}\frac{1}{( 1+| y |_{m}^{n} )}dy}\right|dx}\\
		=&\frac{1}{|B|}\int_B{\left|\int_{\mathbb{H}^n}{\frac{f(\delta_{|x|_h}y)-f_{|y|_hB}}{\left( 1+| y |_{m}^{n} \right)}dy}\right|dx}\\
		\leq& \frac{1}{|B|}\int_B{\left(\int_{\mathbb{H}^n}{\left|\frac{f(\delta_{|x|_h}y)-f_{|y|_hB}}{\left( 1+| y |_{m}^{n} \right)}\right|dy}\right)dx}\\
		\leq& \frac{1}{|B|}\int_{\mathbb{H}^n}{\left(\int_B{\left|\frac{f(\delta_{|x|_h}y)-f_{|y|_hB}}{\left( 1+| y |_{m}^{n} \right)}\right|dx}\right)dy}\\
		=&\frac{1}{|B|}\int_{\mathbb{H}^n}{\left(\int_B{\left|f(\delta_{|x|_h}y)-f_{|y|_hB}\right|dx}\right)\frac{1}{( 1+| y |_{m}^{n} )}dy}\\
		=&\int_{\mathbb{H}^n}{\left(\frac{1}{|B|}\int_B{\left|f(\delta_{|x|_h}y)-f_{|y|_hB}\right|dx}\right)\frac{1}{( 1+| y |_{m}^{n} )}dy}\\
		=&\int_{\mathbb{H}^n}{\left(\frac{1}{|B|}\int_B{\left|f(\delta_{|y|_h}x)-f_{|y|_hB}\right|dx}\right)\frac{1}{( 1+| y |_{m}^{n})}dy}\\
		=&\int_{\mathbb{H}^n}{\left(\frac{1}{||y|_hB|}\int_{|y|_hB}{\left|f(y)-f_{|y|_hB}\right|dy}\right)\frac{1}{( 1+| y |_{m}^{n} )}dy}\\
		\leq& \left\| f \right\| _{BMO( \mathbb{H}^n )}\int_{\mathbb{H}^n}{\frac{1}{( 1+| y |_{m}^{n} )}dy}
	\end{align*}
	hold for $f\in BMO(\mathbb{H}^n)$.
	
	From the supremum of the  $BMO(\mathbb{H}^n)$, it can be derived
	\begin{equation}\label{sufficient}
		\left\| Hf(x)\right\|_{BMO(\mathbb{H}^n)\rightarrow BMO(H^n)}\leq A_{1}^{h},
	\end{equation}
	
	\begin{equation}\label{sufficient 1}
		\left\| Tf(x)\right\|_{BMO(\mathbb{H}^n)\rightarrow BMO(\mathbb{H}^n)}\leq B_{1}^{h},
	\end{equation}
	thus we obtain the boundedness of $H$ on $BMO(\mathbb{H}^n)$.
	
	On the other hand, by taking
	\begin{equation*}\label{example}
		f_0(x)=
		\begin{cases}
			1,       &       {x_{2n+1}> 0},\\
			0,  &   {x_{2n+1}= 0},\\
			-1,  &   {x_{2n+1}< 0}.
		\end{cases}
	\end{equation*}
	Then, $f_0(x)\in {BMO(\mathbb{H}^n)} $ with $\left\| f_o(x)\right\|_{BMO(\mathbb{H}^n)}\ne0$. It is clear that for $f_0(x)$, we have
	\begin{equation}\label{norm_1}
		\|f_0(x)\|_{BMO(\mathbb{H}^n)}=1,
	\end{equation}
	
	$$
	Hf_0(x)=f_0(x)\int_{\mathbb{H}^n}{\frac{1}{\max(1,|y|_{h}^{n})}dy}.
	$$
	and
	$$
	Tf_0(x)=f_0(x)\int_{\mathbb{H}^n}{\frac{1}{( 1+| y |_{m}^{n} )}dy}.
	$$
	Furthermore, using the identity $\eqref{identity}$ and after some simple calculation, we obtain
	\begin{equation}\label{simple _compute}
		\begin{split}
			\left\| Hf_0( x ) \right\| _{BMO( \mathbb{H}^n )}
			=&\left\| f_0( x ) \right\| _{BMO( \mathbb{H}^n )}\int_{H^n}{\frac{1}{\max ( 1,| y |_{h}^{n} )}dy}\\
			=&\int_{\mathbb{H}^n}{\frac{1}{\max ( 1,| y |_{h}^{n} )}dy}.
		\end{split}
	\end{equation}
	Using the identity $\eqref{identity 1}$ and after some simple calculation, we obtain
	\begin{equation}\label{simple _compute 11}
		\begin{split}
			\left\| Tf_0( x ) \right\| _{BMO( \mathbb{H}^n )}
			=&\left\| f_0( x ) \right\| _{BMO( \mathbb{H}^n )}\int_{\mathbb{H}^n}{\frac{1}{( 1+| y |_{m}^{n} )}dy}\\
			=&\int_{\mathbb{H}^n}{\frac{1}{( 1+| y |_{m}^{n} )}dy}.
		\end{split}
	\end{equation}

	Consequently, combing $\eqref{simple _compute}$, $\eqref{simple _compute 11}$ and by using the definition of norm on the Heisenberg group BMO space $BMO(\mathbb{H}^n)$, we can infer that
	\begin{align}\label{norm_2}
		\left\| Hf( x ) \right\| _{BMO( \mathbb{H}^n )}\geq \int_{\mathbb{H}^n}{\frac{1}{\max ( 1,| y |_{h}^{n} )}dy}
	\end{align}
	and
	\begin{align}\label{norm_3}
		\left\| Tf( x ) \right\| _{BMO( \mathbb{H}^n )}\geq \int_{\mathbb{H}^n}{\frac{1}{( 1+| y |_{m}^{n} )}dy}.
	\end{align}
	Combining $\eqref{norm_1}$, $\eqref{norm_2}$ and $\eqref{norm_3}$, it follows that
	\begin{equation}\label{necessary }
		\left\| Hf( x ) \right\| _{BMO( \mathbb{H}^n ) \rightarrow BMO( \mathbb{H}^n )}=\int_{H^n}{\frac{1}{\max ( 1,| y |_{h}^{n} )}dy}
	\end{equation}
	and
	\begin{equation}\label{necessary 1 }
		\left\| Tf\left( x \right) \right\| _{BMO( \mathbb{H}^n ) \rightarrow BMO( \mathbb{H}^n )}=\int_{\mathbb{H}^n}{\frac{1}{( 1+| y |_{m}^{n} )}dy}.
	\end{equation}
	From $\eqref{sufficient}$ , $\eqref{necessary }$ and $\eqref{necessary 1 }$, when $m=1$, it implies that $\eqref{sharp_bound_1}$ and$\eqref{sharp_bound_11}$ holds.

	Next, we show that Theorem 1.1 holds when $m\geq2$.
	
	Using the  above identity $\eqref{identity}$, $\eqref{identity 1}$ and the definition of Heisenberg group BMO space $BMO(\mathbb{H}^{n})$, when $m\geq2$, we can infer that
	\begin{align*}
		&\left( H_m( f_1,\dots,f_m ) ( x ) \right) _B\\
		=&\frac{1}{| B |}\int_B{\left( \int_{\mathbb{H}^{nm}}{\frac{f_1( \delta _{| x |_h}y_1 ) \cdots f_{m}( \delta _{| x |_h}y_{m} )}{\max \left( 1,| y_1 |_{h}^{n},\dots,| y_{m} |_{h}^{n} \right) ^{m}}dy_1\cdots dy_{m}} \right) dx}\\
		=&\int_{\mathbb{H}^{nm}}{\left( \frac{1}{| B |}\int_{B}{f_1( \delta _{| x |_h}y_1 ) \cdots f_{m}\left( \delta _{| x |_h}y_{m} \right) dx} \right) \frac{dy_1\cdots dy_{m}}{\max \left( 1,| y_1 |_{h}^{n},\dots,| y_{m} |_{h}^{n} \right) ^{m}}}\\
		=&\int_{\mathbb{H}^{nm}}{\left( \frac{1}{| B |}\int_{B}{f_1( \delta _{| y_1 |_h}x ) \cdots f_{m}( \delta _{| y_m |_h}x ) dx} \right) \frac{dy_1\cdots dy_{m}}{\max \left( 1,| y_1 |_{h}^{n},\dots,| y_{m} |_{h}^{n} \right) ^{m}}}\\
		=&\int_{\mathbb{H}^{nm}}{\left( \prod_{j=1}^m{\frac{1}{| | y_j |_hB |}\int_{| y_j |_hB}{f_j\left( \delta _{| y_j |_h}x \right) dx}} \right)}\frac{dy_1\cdots dy_{m}}{\max \left( 1,| y_1 |_{h}^{n},\dots,| y_{m} |_{h}^{n} \right) ^{m}}\\
		=&\int_{\mathbb{H}^{nm}}{\frac{1}{\max \left( 1,| y_1 |_{h}^{n},\dots,| y_{m} |_{h}^{n} \right) ^{m}}\prod_{j=1}^m{f_{| y_j |_hB}}}dy_1\cdots dy_{m}
	\end{align*}
	and
	\begin{align*}
		&\left( T_m( f_1,\dots,f_m ) ( x ) \right) _B\\
		=&\frac{1}{| B |}\int_B{\left( \int_{\mathbb{H}^{nm}}{\frac{f_1\left( \delta _{| x |_h}y_1 \right) \cdots f_{m}\left( \delta _{| x |_h}y_{m} \right)}{\left( 1+| y_1 |_{h}^{n}+\cdots +| y_m |_{h}^{n} \right) ^m}dy_1\cdots dy_{m}} \right) dx}\\
		=&\int_{\mathbb{H}^{nm}}{\left( \frac{1}{| B |}\int_{B}{f_1( \delta _{| x |_h}y_1 ) \cdots f_{m}\left( \delta _{| x |_h}y_{m} \right) dx} \right) \frac{dy_1\cdots dy_{m}}{\left( 1+| y_1 |_{h}^{n}+\cdots +| y_m |_{h}^{n} \right) ^m}}\\
		=&\int_{\mathbb{H}^{nm}}{\left( \frac{1}{| B |}\int_{B}{f_1( \delta _{| y_1 |_h}x ) \cdots f_{m}( \delta _{| y_m |_h}x ) dx} \right) \frac{dy_1\cdots dy_{m}}{\left( 1+| y_1 |_{h}^{n}+\cdots +| y_m |_{h}^{n} \right) ^m}}\\
		=&\int_{\mathbb{H}^{nm}}{\left( \prod_{j=1}^m{\frac{1}{\left| | y_j |_hB \right|}\int_{| y_j |_hB}{f_j\left( \delta _{| y_j |_h}x \right) dx}} \right)}\frac{dy_1\cdots dy_{m}}{\left( 1+| y_1 |_{h}^{n}+\cdots +| y_m |_{h}^{n} \right) ^m}\\
		=&\int_{\mathbb{H}^{nm}}{\frac{1}{\left( 1+| y_1 |_{h}^{n}+\cdots +| y_m |_{h}^{n} \right) ^m}\prod_{j=1}^m{f_{| y_j |_hB}}}dy_1\cdots dy_{m}.
	\end{align*}

	Then, using the Fubini theorem again, we get
	\begin{align*}
		&\left\| H_m( f_1,\dots,f_m ) ( x ) \right\| _{BMO( \mathbb{H}^n )}\\
		=&\frac{1}{| B |}\int_B{\left| H( f_1,\dots,f_m ) ( x ) -\left( H( f_1,\dots,f_m ) ( x ) \right) _B \right|dx}\\
		=&\frac{1}{|B|}\int_B{\left| \int_{\mathbb{H} ^{nm}}{\frac{f_1\left( \delta _{|x|_h}y_1 \right) \cdots f_m\left( \delta _{|x|_h}y_m \right)}{\max \left( 1,|y_1|_{h}^{n},\dots ,|y_m|_{h}^{n} \right) ^m}dy_1\cdots dy_m-\int_{\mathbb{H} ^{nm}}{\frac{\prod_{j=1}^m{f_{\left| y_j \right|_hB}}}{\max \left( 1,|y_1|_{h}^{n},\dots ,\left| y_m \right|_{h}^{n} \right) ^m}dy_1\cdots dy_m}} \right|}dx\\
		\leq& \int_{\mathbb{H}^{nm}}{\left( \frac{1}{| B |}\int_B{\left| \prod_{j=1}^m{f_j\left( \delta _{| {y_j}_h |}x \right)}-\prod_{j=1}^m{f_{| y_j |_hB}} \right|dx} \right) \frac{dy_1\cdots dy_m}{\max \left( 1,| y_1 |_{h}^{n},\dots,| y_{m} |_{h}^{n} \right) ^{m}}}\\
		=&\int_{\mathbb{H}^{nm}}{\left( \prod_{j=1}^m{\frac{1}{\left| | y_j |_hB \right|}\int_{| y_j  |_hB}{\left| f_j( y ) -f_{| y_j |_hB} \right|dy}} \right) \frac{dy_1\cdots dy_m}{\max \left( 1,| y_1 |_{h}^{n},\dots,| y_{m} |_{h}^{n} \right) ^{m}}}\\
		\leq& \prod_{j=1}^m{\left\| f_j \right\| _{BMO( \mathbb{H}^n )}}\int_{\mathbb{H}^{nm}}{\frac{dy_1\cdots dy_m}{\max \left( 1,| y_1 |_{h}^{n},\dots,| y_{m} |_{h}^{n} \right) ^{m}}}\\
		=&A_{m}^{h}
	\end{align*}
	and
	\begin{align*}
		&\left\| T_m( f_1,\dots,f_m ) ( x ) \right\| _{BMO( \mathbb{H}^n )}\\
		=&\frac{1}{| B |}\int_B{\left| H\left( f_1,\dots,f_m \right) \left( x \right) -\left( H\left( f_1,\dots,f_m \right) ( x ) \right) _B \right|dx}\\
		=&\frac{1}{|B|}\int_B{\left| \int_{\mathbb{H} ^{nm}}{\frac{f_1(\delta _{|x|_h}y_1)\cdots f_m\left( \delta _{|x|_h}y_m \right)}{\left( 1+|y_1|_{h}^{n}+\cdots +|y_m|_{h}^{n} \right) ^m}dy_1\cdots dy_m-\int_{ \mathbb{H}^{nm}}{\frac{\prod_{j=1}^m{f_{|y_j|_hB}}}{\left( 1+|y_1|_{h}^{n}+\cdots +|y_m|_{h}^{n} \right) ^m}dy_1\cdots dy_m}} \right|}dx\\
		\leq& \int_{\mathbb{H}^{nm}}{\left( \frac{1}{| B |}\int_B{\left| \prod_{j=1}^m{f_j\left( \delta _{\left| {y_j}_h \right|}x \right)}-\prod_{j=1}^m{f_{| y_j |_hB}} \right|dx} \right) \frac{dy_1\cdots dy_m}{\left( 1+| y_1 |_{h}^{n}+\cdots +| y_m |_{h}^{n} \right) ^m}}\\
		=&\int_{\mathbb{H}^{nm}}{\left( \prod_{j=1}^m{\frac{1}{\left| | y_j |_hB \right|}\int_{| y_j |_hB}{\left| f_j( y ) -f_{| y_j |_hB} \right|dy}} \right) \frac{dy_1\cdots dy_m}{\left( 1+| y_1 |_{h}^{n}+\cdots +| y_m |_{h}^{n} \right) ^m}}\\
		\leq& \prod_{j=1}^m{\left\| f_j \right\| _{BMO( \mathbb{H}^n )}}\int_{\mathbb{H}^{nm}}{\frac{dy_1\cdots dy_m}{\left( 1+| y_1 |_{h}^{n}+\cdots +| y_m |_{h}^{n} \right) ^m}}\\
		=&B_{m}^{h}
	\end{align*}
	for every  $( f_1,\dots,f_m ) \in \prod_{j=1}^m{BMO( \mathbb{H}^n )}$.
	
	From the supremum of the  $BMO(\mathbb{H}^n)$, it can be derived
	$$
	\left\| H_m( f_1,...f_m ) ( x ) \right\| _{\prod_{j=1}^m{BMO( \mathbb{H}^n )}\rightarrow BMO( H^n )}\leq A_{m}^{h}
	$$
	and
	$$
	\left\| T_m( f_1,...f_m ) ( x ) \right\| _{\prod_{j=1}^m{BMO( \mathbb{H}^n )}\rightarrow BMO( H^n )}\leq B_{m}^{h}.
	$$
	Then,we can obtain the boundedness of $H_m$ on $BMO(\mathbb{H}^n)$.
	
	On the other hand, by taking
	\begin{equation*}\label{example m}
		f_j(x)=
		\begin{cases}
			1,       &       {x_{2n+1}> 0},\\
			0,  &   {x_{2n+1}= 0},\\
			-1,  &   {x_{2n+1}< 0},
		\end{cases}
	\end{equation*}
	we have $f_j(x)\in {BMO(\mathbb{H}^n)} $ with $\left\| f_j(x)\right\|_{BMO(\mathbb{H}^n)}\ne0$.
	
	We can get that for $j=1,\dots,m$ ,
	\begin{equation}\label{norm_3}
		\left\| f_j( x ) \right\| _{BMO( \mathbb{H}^n )}=1,
	\end{equation}
	
	$$
	H_m( f_1,\dots,f_m ) ( x ) =\prod_{j=1}^m{f_j( x )}\int_{\mathbb{H}^{nm}}{\frac{1}{\max ( 1,| y_1 |_{h}^{n},\dots,| y_m |_{h}^{n} )}dy_1\cdots dy_m}
	$$
	and
	$$
	T_m( f_1,\dots,f_m ) ( x ) =\prod_{j=1}^m{f_j( x )}\int_{\mathbb{H}^{nm}}{\frac{1}{( 1+| y_1 |_{h}^{n}+\cdots +| y_m |_{h}^{n} ) ^m}dy_1\cdots dy_m}.
	$$
	Furthermore, using the identity $\eqref{identity}$, $\eqref{identity 1}$ and some simple calculation, we obtain
	\begin{equation}\label{simple _compute 2}
		\begin{split}
			&\left\| H_m( f_1,\dots,f_m ) ( x ) \right\| _{BMO( \mathbb{H}^n )}\\
			=&\prod_{j=1}^m{\left\| f_j( x ) \right\| _{BMO( \mathbb{H}^n )}}\int_{\mathbb{H}^{nm}}{\frac{1}{\max ( 1,| y_1 |_{h}^{n},\dots,| y_m |_{h}^{n} )}dy_1\cdots dy_m}
		\end{split}
	\end{equation}
	
	\begin{equation}\label{simple _compute 22}
		\begin{split}
			&\left\| T_m( f_1,\dots,f_m ) ( x ) \right\| _{BMO( \mathbb{H}^n )}\\
			=&\prod_{j=1}^m{\left\| f_j( x ) \right\| _{BMO( \mathbb{H}^n )}}\int_{\mathbb{H}^{nm}}{\frac{1}{\left( 1+| y_1 |_{h}^{n}+\cdots +| y_m |_{h}^{n} \right) ^m}dy_1\cdots dy_m}.
		\end{split}
	\end{equation}

	Consequently, combing $\eqref{simple _compute 2}$, $\eqref{simple _compute 22}$ and by using the definition of norm on the Heisenberg group BMO space $BMO(\mathbb{H}^n)$, we can infer that
	\begin{align}\label{norm_4}
		\left\| H_m( f_1,\dots,f_m ) ( x ) \right\| _{BMO( \mathbb{H}^n )}\geq \int_{\mathbb{H}^{nm}}{\frac{1}{\max \left( 1,| y_1 |_{h}^{n},\dots,| y_m |_{h}^{n} \right)}dy_1\cdots dy_m}
	\end{align}
	and
	\begin{align}\label{norm_44}
		\left\| T_m( f_1,\dots,f_m ) ( x ) \right\| _{BMO( \mathbb{H}^n )}\geq \int_{\mathbb{H}^{nm}}{\frac{1}{\left( 1+| y_1 |_{h}^{n}+\cdots +| y_m |_{h}^{n} \right) ^m}dy_1\cdots dy_m}.
	\end{align}
	Combining $\eqref{norm_3}$, $\eqref{norm_4}$ and $\eqref{norm_44}$ it follows that
	\begin{equation}\label{necessary 2}
		\| H_m( f_1,\dots,f_m ) ( x ) \| _{BMO( \mathbb{H}^n )}=\int_{\mathbb{H}^{nm}}{\frac{1}{\max \left( 1,| y_1 |_{h}^{n},\dots,| y_m |_{h}^{n} \right)}dy_1\cdots dy_m},
	\end{equation}
	
	\begin{equation}\label{necessary 22}
		\left\| T_m( f_1,\dots,f_m ) ( x ) \right\|_{BMO( \mathbb{H}^n )}=\int_{\mathbb{H}^{nm}}{\frac{1}{\max ( 1,| y_1 |_{h}^{n},\dots,| y_m |_{h}^{n} )}dy_1\cdots dy_m}.
	\end{equation}
	From $\eqref{necessary 2}$ and $\eqref{necessary 22}$, when $m\geq2$, it implies that $\eqref{sharp_bound_1}$ holds.

	This completes the proof of Theorem $\ref{main_1}$ and $\ref{main_11}$.

\end{proof}

\section{Further sharp bound for Hausdorff operator on Heisenberg group BMO space}

In this section, we will extend the boundedness of $m$-linear $n$-dimensional Hausdorff operators to generalized integral operator on on Heisenberg group BMO space.

\begin{definition}
	Let $m$ be a positive integer , $f_1,\dots,f_m$ be nonnegative locally integrable
	functions on $\mathbb{H}^n$. The $m$-linear $n$-dimensional Hausdorff operator  is defined by
	\begin{equation}
		G_{\varPhi}\left( f_1,\dots ,f_m \right) (x)=\int_{\mathbb{H} ^{nm}}{\frac{\varPhi \left( \frac{x}{\delta _{|y_1|_h}},\dots ,\frac{x}{\delta _{|y_m|_h}} \right)}{|y_1|_{h}^{n}\cdots |y_m|_{h}^{n}}f_1(y_1)\cdots}f_m(y_m)dy_1\cdots dy_m.
	\end{equation}
	where $x\in \mathbb{H}^n\backslash\left\{ 0\right\}$.
\end{definition}
For convenience, it is need to make a change of variable, then it follows that
\begin{equation}\label{identity 2}
	G_{\varPhi}(f_{1},\ldots,f_{m})(x)=\int_{\mathbb{H}^{nm}}{\frac{\varPhi ( y_1,\dots,y_m )}{| y_1 |_{h}^{n}\cdots | y_m |_{h}^{n}}f_1( x\delta _{| y_1 |_{h}^{-1}} ) \cdots}f_m( x\delta _{| y_m |_{h}^{-1}} ) dy_1\cdots dy_m.
\end{equation}
Using the same proof method, we can get
$$
\left\|G _{\varPhi}( f_1,\dots,f_m ) \right\| _{BMO( \mathbb{H}^n )}\leq \prod_{j=1}^m{\left\| f_j \right\| _{BMO\left(\mathbb{H}^n \right)}}\int_{\mathbb{H}^{nm}}{\frac{\varPhi \left( y_1,\dots,y_m \right)}{| y_1 |_{h}^{n}\cdots | y_m |_{h}^{n}}}dy_1\cdots dy_m.
$$

Using the similar method of the  proof of Theorem $\ref{main_1}$, we give the following result.

\begin{thm}\label{main_12}
	Let $m\in \mathbb{N}$,  if
	
	\begin{equation}\label{condition_5}
		F_{m}^{h}:=\int_{\mathbb{H}^{nm}}{\frac{\varPhi ( y_1,\dots,y_m )}{| y_1 |_{h}^{n}\cdots | y_m |_{h}^{n}}}dy_1\cdots dy_m< \infty,
	\end{equation}
	then $T_{m}^{h}$ is bounded from $BMO(\mathbb{H}^{n})\times\cdots \times BMO(\mathbb{H}^{n})$ to $BMO(\mathbb{H}^{n})$.
	\begin{equation}\label{sharp_bound_5}
		\|G_{\varPhi}\|_{\prod\limits_{j=1}^{m}BMO(\mathbb{H}^{n})\rightarrow BMO(\mathbb{H}^{n})}=F_{m}^{h}.
	\end{equation}
\end{thm}


\newpage

\begin{flushleft}

	\vspace{0.3cm}\textsc{Huan Liang\\School of Science\\Shandong Jianzhu University \\Jinan, 250000\\P. R. China}
	
	\emph{E-mail address}: \textsf{lianghuan202211@163.com}

	\vspace{0.3cm}\textsc{Xiang Li\\School of Science\\Shandong Jianzhu University\\Jinan, 250000\\P. R. China}
	
	\emph{E-mail address}: \textsf{lixiang162@mails.ucas.ac.cn}

	\vspace{0.3cm}\textsc{Dunyan Yan\\School of Mathematical Sciences\\University of Chinese Academy of Sciences\\Beijing, 100049\\P. R. China}
	
	\emph{E-mail address}: \textsf{ydunyan@163.com}

\end{flushleft}

\end{document}